\documentclass[11pt]{amsart}

\usepackage[top=1in, bottom=1in, left=1in, right=1in]{geometry}

\usepackage{graphicx}
\usepackage[dvipsnames]{xcolor}
\usepackage{amsfonts,amsmath,amssymb,amsthm,bbm}
\usepackage{enumitem}
\usepackage[normalem]{ulem}
\usepackage[hidelinks]{hyperref}
\usepackage{mathtools}
\usepackage{multirow}
\usepackage{comment}
\usepackage{booktabs}
\usepackage{multicol}

\newtheorem{theorem}{Theorem}[section]
\newtheorem{proposition}[theorem]{Proposition}
\newtheorem{lemma}[theorem]{Lemma}
\newtheorem{corollary}[theorem]{Corollary}

\theoremstyle{remark}
\newtheorem{remark}[theorem]{Remark}

\theoremstyle{definition}

\theoremstyle{plain} 
\newcommand{\thistheoremname}{}
\newtheorem*{genericthm*}{\thistheoremname}
\newenvironment{namedthm*}[1]
  {\renewcommand{\thistheoremname}{#1}%
   \begin{genericthm*}}
  {\end{genericthm*}}

\newcommand{\N}{\mathbb N}

\newcommand{\Z}{\mathbb Z}     

\newcommand{\bigo}{\mathcal{O}}

\newcommand{\geo}[1]{\text{Geo}(#1)}
\newcommand{\Var}[1]{\text{Var}(#1)} 
\newcommand{\ber}[1]{\text{Ber}(#1)}
\newcommand{\rj}[2]{\mathcal{E}_{#1}^{#2}}
\newcommand{\lj}[2]{\mathcal{D}_{#1}^{#2}}
\newcommand{\cala}{\mathcal{A}} 
\newcommand{\calb}{\mathcal{B}}
\newcommand{\calc}{\mathcal{C}}

\newcommand{\fl}[1]{\lfloor #1 \rfloor}

\newcommand{\ind}[1]{\mathbbm{1}_{\{#1\}}}

\title[A New Branching Process with Applications for a Self-Interacting Random Walk]{A Branching Process with Geometric Emigration and the ``Have Your Cookie and Eat It" Random Walk}

\author{Patrick Carper}
\address{Patrick Carper, Purdue University, Department of Mathematics, 150 N. University Street,
West Lafayette, IN 47907, USA}
\email{pcarper@purdue.edu}

\date{\today}

\begin{document}

\begin{abstract}
We consider a random walk on the integers that at each site is biased to step right until the first time it has stepped left from that site, thereafter stepping left or right from that site with equal chance, introduced in \cite{Pin10}. In that paper, the author calculated the speed of the walk and the probability that such a walk escapes to infinity for some ranges of \(p\) and conjectured these formulas hold over a wider range. To study this walk, we introduce a branching process featuring geometric emigration and answer some questions concerning its life-periods, transience and recurrence, and the mean of its stationary distribution when positive recurrent. We apply these results to confirm the predictions in \cite{Pin10}. For another application, we calculate the speed and escape probability of a related self-interacting random walk.
\end{abstract}

\maketitle

\section{Introduction}
\emph{Cookie random walks} are a well-studied model of self-interacting random walks with numerous variations, first considered in \cite{BW03}. In an \((M,p)\)-cookie random walk on \(\Z\), \(M\) cookies of strength \(p\in(1/2,1)\) are placed at each site \(x\in\Z\). The walk starts at 0. Upon visiting \(x\), if there is a cookie there, the walker eats it and steps right with probability \(p\) and left with probability \(1-p\). At cookie-less sites (those that have been visited at least \(M\) times) it steps right and left with probability \(1/2\). In \cite{Pin10}, Pinsky proposed the following variation: At each site, there is initially one cookie with strength \(p\in(1/2,1)\). The walk begins at the origin and, at each site, steps right from it with probability \(p\). Only if the walk elects to \emph{step left} from a site does it eat the cookie there. On visits to cookie-less sites it steps left and right with probability \(1/2\). Pinsky christened this the ``have your cookie and eat it" (HYC) random walk and noted that the walk seems to behave like a cookie random walk with \(M=(1-p)^{-1}\) cookies, even though \((1-p)^{-1}\) isn't typically an integer. That is, one can predict the ranges of \(p\) for which the walk is transient and ballistic by naively treating it like a \(((1-p)^{-1},p)\)-cookie random walk and applying results known for that process. This is not entirely a coincidence. The walk can be framed as an excited random walk with a Markovian cookie stack, see the appendix for details. Criteria for which values of \(p\) the walk is recurrent, transient but with zero speed, and ballistic, which were already provided by Pinsky, can also be deduced from the results in \cite{KP17} and \cite{KMP21}, along with limit laws for the location of the walk and functional limit laws for its rescaled path. Here we will study a branching process that provides simple formulas for the speed and escape probability of both this random walk some related processes.

We denote probabilities for the walk started from \(x\) by \(P_x\). Also, let \(T_x=\inf\{n\geq 0:X_n=x\}\). The support of any Geometric random variable appearing is \(\N_0=\{0,1,2,\ldots\}\).

\subsection{Known results and our contributions}
The following results were proved in \cite{Pin10}.

\begin{theorem}[Recurrence/transience/escape probability]\label{thm recurrence}Consider a HYC walk...
 \begin{itemize}
     \item[(i)] if \(p\leq 2/3\), \(P_1(T_0=\infty)=0\) and the walk is recurrent, i.e. with probability 1 it returns to each point infinitely often.
     \item[(ii)] if \(p>2/3,\ P_1(T_0=\infty)>0\) and the walk is transient to the right, i.e. \(\lim_{n\to\infty}X_n=\infty\) a.s. In particular \(P_1(T_0=\infty)=\frac{3p-2}{2p-1}\) for \(p\geq 3/4\) and the author predicted this applies for \(p> 2/3\).
 \end{itemize}
\end{theorem}
\begin{theorem}[Speed]\label{pinsky speed}For the HYC walk, there exists a deterministic speed \(v(p)\) such that \(v(p)=\lim_{n\to\infty}\frac{X_n}{n}\) a.s. and
\begin{itemize}
    \item[(i)] if \(p<3/4\), then \(v(p)=0\).
    \item[(ii)] if \(p>3/4\), then \(v(p)>0\) and furthermore \(v(p)=4p-3\) for \(p>4/5\) and the author predicted this holds for \(p \geq 3/4\).\footnote{The theorem is stated for \(p>10/11\) in \cite{Pin10}, but \cite{Let21} found the proof applies for all \(p>4/5\).}
\end{itemize}
\end{theorem}
Our first two results confirm Pinsky's predictions that the formulae for the escape probability and the speed of the walk hold over the entire range where they make sense.
 \begin{theorem}[Escape Probability]\label{thm hyc escape}
     If \(p>2/3\), \(P_1(T_0=\infty)=\frac{3p-2}{2p-1}\)
 \end{theorem}
 \begin{theorem}[Speed]\label{thm hyc speed}
     If \(p\geq 3/4\), \(v(p)=4p-3\).
 \end{theorem}

 \subsection{Overview of the paper}
 In Section 2, we will discuss how the number of times the walk has stepped left and right from each nonnegative site by the first time it hits site \(n\) can be viewed as a branching process. We observe that this branching process has some novel features that prevent it from satisfying the assumptions of models that have been previously treated in the branching process literature, and so in Section 3 we introduce and study a more general model having these features. In Section 3.1, we obtain tail asymptotics on the \emph{life-period} of this process---the amount of time it spends away from zero---which, in turn, determine its recurrence and transience. In Section 3.2, we calculate the mean of the stationary distribution as well as the weight it assigns to 0 in the positive recurrent regime. Section 4 is dedicated to the proofs of our main results for the branching process. In Section 5, we use these results to compute the speed and escape probability of a HYC walk, proving Theorems \ref{thm hyc escape} and \ref{thm hyc speed}. Finally, in Section 6, we introduce a related self-interacting random walk whose speed and escape probability are easily calculated as a further application of our formulae.

\section{Reduction to branching processes}
In this section we will construct a branching processes that counts the number of times the walk has stepped left from each site by the first time the walk hits site \(n\). The idea is well-known: \cite{KKS75} applies it to find limit laws for a random walk in a random environment, \cite{BS08-2} uses it to find limit laws for transient but with zero speed \((M,p)\)-cookie random walks, and \cite{KMP21} relies on it to find functional limit laws for the rescaled path of a cookie random with a Markovian cookie stack. First, by Theorem \ref{thm recurrence}, for any \(p> 1/2\), we a.s. have \(T_n<\infty\) for all \(n\geq 0\). Let
\begin{align*}
    \rj{x}{n}=\sum_{k=0}^{T_n-1}\ind{X_k=x,X_{k+1}=x+1}\quad\text{and}\quad\lj{x}{n}=\sum_{k=0}^{T_n-1}\ind{X_k=x,X_{k+1}=x-1}
\end{align*}
which count, respectively, the number of right steps and left steps from site \(x\) up to time \(T_n\). We can frame the sequence \(\{\lj{n}{n},\lj{n-1}{n},\ldots,\lj{1}{n},\lj{0}{n}\}\) as a branching process with migration. Set \(V_0=0\) and for \(n\geq 0\)
\begin{align*}
    V_{n+1}=\Bigl(1+\sum_{m=1}^{V_n-K_n+1}X_{m,n}\Bigr)\ind{K_n\leq V_n}.
\end{align*}
where \(\{X_{m,n}\}_{m\geq 1,n\geq 0}\) are independent \(\geo{1/2}\) random variables and \(K_n\) are independent \(\geo{1-p}\) random variables that are independent of the \(X_{m,n}\). We claim that \[\{\lj{n}{n},\ldots,\lj{0}{n}\}\overset{d}{=}\{V_0,\ldots,V_n\}.\] Of course \(\lj{n}{n}=0=V_0\). The moment the walk hits \(n\), the walk has stepped right from \(x\) to \(x+1\) \(\rj{x}{n}\) times. The \(\rj{x}{n}\)-th step from \(x\) is to the right. The number of right steps from \(x\) before the first left step is the number of failures before the first success in a Bernoulli sequence with success probability \(1-p\) if this success occurs before the \(\rj{x}{n}\)-th right step from \(x\) and zero otherwise. Call the number of trials before this success occurs \(K_n\). If \(K_n\leq \rj{x}{n}-1\), there is one left jump and then the site becomes neutral, so that the number of left jumps from \(x\) to \(x-1\) between each of the \(\rj{x}{n}-K_n\) additional right jumps from \(x\) to \(x+1\) is counted by the number of failures before the first success in a Bernoulli sequence with success probability \(1/2\), and are thus each \(\geo{1/2}\). So conditioned on \(\rj{x}{n}\),
\[\lj{x}{n}\overset{d}{=}\Bigl(1+\sum_{m=1}^{\rj{x}{n}-K_n}X_{m,n}\Bigr)\ind{K_n\leq \rj{x}{n}-1}.\]
Now for \(0\leq x\leq n-1\), \(\rj{x}{n}=\lj{x+1}{n}+1\), since for the walk to be at site \(n\), it must have crossed from \(x\) to \(x+1\) one more time than it crossed from \(x+1\) to \(x\). The claim follows from this substitution.
 
\section{A branching process with two types of immigration and geometric emigration}
Because the branching process introduced in the previous section features geometric emigration, it fails to fit into the classic framework of branching processes with immigration and bounded emigration in \cite{YY95}. It does fit into the newer framework of \textit{multiple controlled branching process}, see \cite[Definition 3.2]{GdPY18} and can also be obtained by summing the components of a controlled \textit{multi-type} branching process (CMBP), a branching process living in \(\N_0^d\) whose transitions are governed by certain control functions, see \cite{BGMCdP24} for a complete definition. Results concerning the extinction time for some CMBP were established in the univariate case \(d=1\) in \cite{GMdP05}; however, a theorem that provides extinction times of CMBP in \(\N_0^d\), \(d\geq 2\), appears absent from the literature. Our branching process \(V_n\) \emph{is} a special case of the \emph{branching-like processes} studied in \cite{KP17}. A version of our Theorem \ref{thm tau} specialized to the process \(V_n\) can be partially recovered from \cite[Theorem 2.7]{KP17}.

We will confine ourselves to the following process that generalizes \(V_n\):
\begin{equation}\label{def yn}
    Y_{n+1}=\Bigl(J_n+\sum_{m=1}^{Y_n+I_n-K_n}X_{m,n}\Bigr)\ind{K_n<Y_n+I_n}
\end{equation}
where \(Y_0\) is given, \(I_n,J_n,K_n\in \N_0\) are random variables that represent immigrants arriving before reproduction, immigrants arriving after reproduction, and emigrants leaving before reproduction. We assume \(I_n,J_n\), and \(K_n\) are iid copies of the random variables \(I,J\), and \(K\), respectively. Suppose the random variables \(X_{m,n}\in \N_0\) are iid copies of the random variable \(X\), which is independent of \(I,J\), and \(K\). Finally, \(I_n,J_n,K_n\), and \(X_{m,n}\) are all jointly independent for all \(n\). Unlike the process \(V_n\), \(Y_n\) cannot be written as one of the branching-like process in \cite{KP17} for generic random variables \(I, J,\) and \(X\). Finally, throughout we will assume\footnote{We make the assumption \(P(J>0)=1\) for convenience. It implies the only time our branching process hits zero is when it is wiped out by emigration. We believe the methods here together with some additional arguments can be used to prove analogous results for the case when \(P(J=0)>0\), but this introduces another term in Lemma \ref{zn recurrence} and Corollary \ref{y functional eq}, and we anticipate it will complicate the proofs of our main results.} \(P(J>0)=1\) and \(K\sim \geo{1-p}\) for \(0<p<1\), i.e. \(P(K=k)=(1-p)p^k\) for all \(k\geq 0\). For most results, we will need one of the two following sets of additional assumptions:
\begin{multicols}{2}
  {\bfseries Assumption A.}
    \begin{itemize}
        \item  \(E[J^2],E[I^2]<\infty\)
        \item \(E[X]=1\) and  \(E[X^3]<\infty\)
    \end{itemize}
    \columnbreak
   \hspace{0.5 cm}{\bfseries Assumption B.}
    \begin{itemize}
    \item \(E[J^3],E[I^3]<\infty\)
    \item \(E[X]=1\) and  \(E[X^4]<\infty\)
    \end{itemize}   
\end{multicols}
\vspace{-0.5 cm}
\[\text{Define also}\quad\quad E[I]=\mu,\ E[J]=\nu,\ E[K]=p(1-p)^{-1}=\rho,\ 0<\Var{X}=\sigma^2,\ E[X^3]=\alpha.\] 
We will need to define many probability generating functions throughout. To remind the reader which random variables they determine, we set \(Y_n(t)=E[t^{Y_n}],\ J(t)=E[t^J],\ I(t)=E[t^I],\) and  \(K(t)=E[t^K]=(1-p)(1-pt)^{-1}.\) Adhering to convention, take the offspring probability generating function to be \(F(t)=E[t^X]\).
\subsection{Results concerning recurrence, transience, and life-periods}
Note that if \(P(I=0)=1\), 0 is an absorbing state for the Markov chain \(Y_n\). Define
\(\theta=\frac{2}{\sigma^2}(\mu+\nu-\rho).\) We have the following theorem:
\begin{theorem}\label{thm markov}
    Assume \(P(I=0)<1\), the Markov chain \(Y_n\) is irreducible, and Assumption A holds. Then \(Y_n\) can be classified as
    \begin{itemize}
    \item Transient, if \(\theta>1\).
    \item Null-recurrent, if \(0<\theta\leq 1\) and if \(\theta=0\) and Assumption B holds.
    \item Positive recurrent, if \(\theta<0\).
\end{itemize}
\end{theorem}
Related are results about the \emph{life-period} \(\tau\) of the process \(Y_n\) with \(Y_0=0\). Let \(T_+=\inf_{n\geq 1}\{Y_n>0\}\) and \(T_0=\inf_{n\geq 1}\{Y_n=0\}\). Then define \(\tau=\inf_{n\geq 1}\{Y_{n+T_+}=0\}\).
\begin{theorem}\label{thm tau}
    Assume \(P(I=0)<1\) and Assumption A. There exists positive, explicit constants \(c_1,c_2,\) and \(c_3\) such that
    \begin{itemize}
    \item \(P(\tau>n)\sim c_1\), if \(\theta>1\).
    \item \(P(\tau>n)\sim c_2/\log(n)\), if \(\theta=1\).
    \item \(P(\tau>n)\sim c_3/n^{1-\theta}\), if \(0<\theta<1\).
    \end{itemize}
    If \(\theta=0\) and Assumption B holds, then \(P(\tau>n)\sim \frac{2}{\sigma^2}(E[Y_{T_+}]+\mu)n^{-1}\). Lastly, \(E[\tau]=1-\frac{2}{\theta \sigma^2}(E[Y_{T_+}]+\mu)\) when \(\theta<0\).
\end{theorem}
Theorem \ref{thm markov} follows easily from Theorem \ref{thm tau} and the identity \(P(T_0>n|Y_0=0)=P(\tau\geq n)P(Y_1>0|Y_0=0)\) for \(n\geq 1\).\footnote{If we condition on \(Y_0=y>0\), then our proof of Theorem \ref{thm tau} will show that the same estimates on \(P(\tau>n)\) apply to \(P(T_0>n)\) (with different constants). The assumption \(P(I=0)<1\) is no longer necessary in this case.}
\subsection{Results concerning the stationary distribution}
Theorem \ref{thm markov} classifies the Markov chain \(Y_n\) as positive recurrent when \(\theta<0\). Let \(\pi\) denote the stationary distribution of \(Y_n\), which we assume to be irreducible and aperiodic, and \(\Bar{\pi}=\sum_{k\geq 0}k\pi(k)\). We have the following results:
\begin{theorem}\label{thm stationary mean}
    Under Assumption B, if \(-1\leq \theta<1,\ \Bar{\pi}=\infty\). If \(\theta<-1\) then \(\Bar{\pi}<\infty\) and
    \[\Bar{\pi}=\frac{-\theta\mu}{4(\theta+1)(\nu-\rho)}\left(\theta\sigma^2-2\mu+\frac{4\nu}{\theta}+\frac{2\Var{I}}{\mu}-\frac{4\Var{I+J+K}}{\sigma^2\theta}\right).\]
\end{theorem}
\begin{theorem}\label{pi 0}
    Under Assumption A, if \(\theta<0\),
    \[\pi(0)=\frac{\mu+\nu-\rho}{\nu-\rho}.\]
\end{theorem}

\section{Proofs of the main results for the branching process}
\subsection{Proofs of the claims concerning recurrence, transience, and life-periods} Define the process stopped at 0 as \(Z_0\overset{d}{=}Y_{T_+}\) and for \(n\geq 0\) as
\[Z_{n+1}\overset{d}{=}Y_{n+1+T_+}\ind{Z_n>0}.\]
We have \(P(Z_n>0)=P(\tau>n)\) and for \(n\geq 0\),
\[Z_{n+1}=\Bigl(J_n+\sum_{m=1}^{Z_n+I_n-K_n}X_{m,n}\Bigr)\ind{K_n<Z_n+I_n,Z_n>0}.\]
Let \(Z_n(t)=E[t^{Z_n}]\) to be the probability generating function of \(Z_n\). Our first lemma is a recurrence for \(Z_n(t)\):
\begin{lemma}\label{zn recurrence}
    For \(n\geq 0\), for any \(t\) with \(F(t)\neq p\), \[Z_{n+1}(t)=Z_n(0)A(t)+Z_n(p)B(t)+Z_n(F(t))C(t)\]
    with, recalling \(K(t)=(1-p)(1-pt)^{-1}\),
    \begin{align*}
        A(t)&=1-I(p)-J(t)K\left((F(t))^{-1}\right)\left(I(F(t))-I(p)\right),\\
        B(t)&=I(p)\left(1-J(t)K\left((F(t))^{-1}\right)\right),\quad\text{and}\quad C(t)=J(t)I(F(t))K\left((F(t))^{-1}\right).\\
    \end{align*}
\end{lemma}
\begin{proof}
Since \(J_n\) is strictly positive,
\begin{align*}
    P(Z_{n+1}=0|Z_n)&=\ind{Z_n=0}+\ind{Z_n>0}P(K_n\geq Z_n+I_n|Z_n)\\
    &=\ind{Z_n=0}(1-P(K\geq I))+P(K\geq Z_n+I|Z_n)\\
    &=\ind{Z_n=0}\Bigl(1-\sum_{i\geq 0}P(I=i)P(K\geq i)\Bigr)+\sum_{k\geq 0}P(I=i)P(K\geq i+Z_n|Z_n)\\
    &=\ind{Z_n=0}\Bigl(1-\sum_{i\geq 0}P(I=i)p^i\Bigr)+p^{Z_n}\sum_{i\geq 0}P(I=i)p^i\\
    &=\ind{Z_n=0}(1-I(p))+p^{Z_n}I(p)
\end{align*}
and so taking expectations yields
\[P(Z_{n+1}=0)=Z_n(0)(1-I(p))+Z_n(p)I(p)\]
Now let \(x\geq1\). Set \(S_k\overset{d}{=}\sum_{m=1}^kX_{m,n}\). Similar manipulations yield
\[P(Z_{n+1}=x|Z_n)=\ind{Z_n>0}\sum_{i\geq 0}P(I=i)\sum_{k=0}^{Z_n+i-1}P(K=k)\sum_{j=0}^xP(J=j)P(S_{Z_n+i-k}=x-j).\]
Since \(J\) is positive, this expression is 0 when \(x=0\), so to compute \(\sum_{x\geq 1}P(Z_{n+1}=x|Z_n)t^x\) using our above expression, we may sum from 0. After an application of Fubini, this sum simplifies as
\begin{align*}
&\ind{Z_n>0}\sum_{i\geq 0}P(I=i)\sum_{j\geq 0}P(J=j)t^j\sum_{k=0}^{Z_n+i-1}P(K=k)\sum_{x\geq j}t^{x-j}P(S_{Z_n+i-k}=x-j)\\
&=\ind{Z_n>0}(1-p)J(t)\sum_{i\geq 0}P(I=i)\sum_{k=0}^{Z_n+i-1}p^k(F(t))^{Z_n+i-k}\\
&=\ind{Z_n>0}(1-p)J(t)\left(\frac{(F(t))^{Z_n}I(F(t))-p^{Z_n}I(p)}{1-p/F(t)}\right)\\
&=(1-p)J(t)\left(\frac{(F(t))^{Z_n}I(F(t))-p^{Z_n}I(p)}{1-p/F(t)}\right)-\ind{Z_n=0}(1-p)J(t)\left(\frac{I(F(t))-I(p)}{1-p/F(t)}\right),
\end{align*}
where we used \(F(t)\neq p\) to get from the second to the third line. Taking expectations here yields \(\sum_{x\geq 1}P(Z_{n+1}=x)t^x\) equals
\[(1-p)J(t)\left(\frac{Z_n(F(t))I(F(t))-Z_n(p)I(p)}{1-p/F(t)}\right)-Z_n(0)(1-p)J(t)\left(\frac{I(F(t))-I(p)}{1-p/F(t)}\right).\]
Combining this with the expression for \(P(Z_{n+1}=0)\) and rearranging terms produces the claim.
\end{proof}
\begin{remark}\label{a+b+c=1}
    Let \(A(t),\ B(t),\) and \(C(t)\) be as above. Then \(A(t)+B(t)+C(t)=1\) for any \(t\) with \(F(t)\neq p\).
\end{remark}
In the next several pages, we will use the above recurrence to derive an expression for the generating function of \(P(Z_n>0)\). Recursing on the final term in Lemma \ref{zn recurrence}, we will find an expression for \(Z_{n+1}(t)\) in terms of two convolutions and a third term. Evaluation of this expression at \(t=0\) and \(t=p\) yields recurrences for \(Z_n(0)\) and \(Z_n(p)\) in terms of convolutions of their previous values with two other sequences. Since the generating function of a convolution of two sequences is just the product of those sequences' respective generating functions, defining several generating functions will convert the two recurrences into a system of two equations with two unknowns, which we solve to find an expression for the generating function of \(P(Z_n>0)\). To proceed, we will need considerably more notation. Let \(F_{0,t}=t\) and, for \(n\geq 0\), set \(F_{n+1,t}=F(F_{n,t})\). \(C(t)\) is not defined for \(t\) with \(F(t)=p\). Assuming \(F_{k,t}\neq p\) for \(1\leq k\leq n\), set \(\Pi_{-1,t}=1\) and for \(n\geq 0,\ \Pi_{n,t}=\prod_{k=0}^nC(F_{k,t})\). Let \(G(t)=Z_0(t)=E[t^{Z_0}]=E[t^{Y_{T^+}}]\). Lastly, set
\begin{align*}
    &\cala(t)=\sum_{n\geq 0}A(F_{n,p})\Pi_{n-1,p}t^n\qquad\calb(t)=\sum_{n\geq 0}B(F_{n,p})\Pi_{n-1,p}t^n\\
    &\calc(t)=\sum_{n\geq 0}G(F_{n+1,p})\Pi_{n,p}t^n\qquad H(t)=\sum_{n\geq 0}(1-I(p))^nI(p)^n=I(p)(1-(1-I(p))t)^{-1}.
\end{align*}
Lemma \ref{several expansions} will establish \(\cala(t),\calb(t),\) and \(\calc(t)\) are well-defined for \(0\leq t<1\).
\begin{lemma}\label{q gf}
Let \(Q(t)=\sum_{n\geq 0}P(Z_n>0)t^n\).  Then, for all \(t\) sufficiently close to 1, \(Q(t)=(1-t)^{-1}-U(t)/V(t)\) with
\[V(t)=1-t\calb(t)-t^2H(t)\cala(t)\quad\text{and}\quad U(t)=tH(t)(G(p)+t\calc(t))\]
\end{lemma}
\begin{proof}
Assume \(F_{k,t}\neq p\) for all \(k\geq 1\). Recursing on only the final term in Lemma \ref{zn recurrence} produces
\begin{equation}\label{zn recurrence 2}
    Z_{n+1}(t)=\sum_{k=0}^nZ_{n-k}(0)A(F_{k,t})\Pi_{k-1,t}+\sum_{k=0}^nZ_{n-k}(p)B(F_{k,t})\Pi_{k-1,t}+G(F_{n+1,t})\Pi_{n,t}.
\end{equation}
Now the assumption \(\sigma^2>0\) implies \(F(x)>x\) for all \(0\leq x<1\), and since \(F(x)\) is increasing, we thus have \(F_{k,p}>p\) for all \(k\geq 1\). This implies \(\Pi_{n,p}\) is well-defined and positive for all \(n\). Moreover, we are free to evaluate \ref{zn recurrence 2} at \(t=p\).
Define the generating functions \(R(t)=\sum_{n\geq 0}Z_n(0)t^n\) and \(P(t)=\sum_{n\geq 0}Z_n(p)t^n\), which are a priori well-defined for \(0\leq t<1\). Evaluating the above recurrence at \(p\) then taking generating functions on both sides of the equation gives
\begin{equation}\label{p}
P(t)=G(p)+tR(t)\cala(t)+tP(t)\calb(t)+t\calc(t)
\end{equation}
since \(Z_0(p)=G(p)\). Now recall from the proof of Lemma \ref{zn recurrence} that
\[P(Z_{n+1}=0)=Z_n(0)(1-I(p))+Z_n(p)I(p).\]
Since \(Z_0(0)=0\), this unwinds as
\[Z_{n+1}(0)=\sum_{k=0}^nZ_{n-k}(p)(1-I(p))^kI(p),\]
and so
\begin{equation}\label{r}
R(t)=tH(t)P(t).
\end{equation}
When \(1-t\calb(t)\neq 0\) and \(V(t)\neq 0\), we may solve the system of equations \eqref{p} and \eqref{r}  to get
\[R(t)=\frac{U(t)}{V(t)}\]
with \(U(t)\) and \(V(t)\) as above. When \(|\calb(1)|<\infty\), our proof of Theorem \ref{thm tau} will show \(\calb(1)<1\), and so \(1-t\calb(t)\) is nonzero for all \(t\) sufficiently close to 1. We will also find leading order asymptotics for \(V(t)\) in that proof that will easily imply \(V(t)\neq 0\) for all \(t<1\) but sufficiently close. Finally, \(Q(t)=(1-t)^{-1}-R(t)\).
\end{proof}
The following proposition establishes the asymptotics of sequence \(\{F_{n,t}\}_{n\geq 0}\). The proof is a straightforward adaptation of de Bruijn's argument and is provided in the appendix.
\begin{proposition}\label{fn asymptotics}
    Let \(0\leq t<1\). If \(E[X^3]<\infty\), \(F_{n,t}=1-\frac{2}{\sigma^2}n^{-1}+\bigo(n^{-2}\log(n))\). If \(E[X^4]<\infty\), \(F_{n,t}=1-\frac{2}{\sigma^2}n^{-1}+\omega n^{-2}\log(n)+C_tn^{-2}+\bigo(n^{-3}\log^2(n))\) where \(\omega=\frac{2(3\sigma^4+6\sigma^2-2\alpha+2)}{3\sigma^6}\) and \(C_t\) is a constant dependent on \(t\).
\end{proposition}
We will need the following estimate on \(\Pi_{n,t}\):
\begin{lemma}\label{lemma prod c asymptotics}
    Let \(p\leq t<1\) and suppose Assumption A holds. There exists a constant \(c_t>0\) dependent on \(t\) such that \[\Pi_{n,t}= c_tn^{-\theta}(1+\bigo(n^{-1}\log(n)).\] If Assumption B holds, \[\Pi_{n,t}=c_tn^{-\theta}\left(1-\omega C'(1)n^{-1}\log(n)-\beta n^{-1}+\bigo(n^{-2}\log^2(n))\right)\]
    with \(\beta=\frac{2}{\sigma^4}C''(1)+(C_t+\omega)C'(1)+\frac{\theta(1-\theta)}{2}\) and \(C'(1)=\mu+\nu-\rho\), where \(C_t\) is another constant dependent on \(t\) and \(\omega\) is as in Proposition \ref{fn asymptotics}.
\end{lemma}
\begin{proof}
    We will prove the result under Assumption B, the proof under Assumption A being similar. Recall \(C(x)=(1-p)J(x)I(F(x))(1-p/F(x))^{-1}=J(x)I(F(x))K\left((F(x))^{-1}\right)\). Since \(I,J\), and \(X\) have finite third moments, \(C'''(1)<\infty\). Now \(C(1)=1\) and thus
    \[C(x)=1-(1-x)C'(1)+\frac{(1-x)^2}{2}C''(1)+\bigo(1-x)^3\] as \(x\uparrow 1\). It is easy to compute \(\theta=\frac{2}{\sigma^2}C'(1)\). Applying this and Proposition \ref{fn asymptotics}, we have
    \[C(F_{n,t})=1-\theta n^{-1}+\omega C'(1)n^{-2}\log(n)+\Bigl(\frac{2}{\sigma^4}C''(1)+C_tC'(1)\Bigr)n^{-2}+\bigo(n^{-3}\log^2(n)).\]
    Using \(\log(1-x)=-x-\frac{x^2}{2}+\bigo(x^3)\), we find
    \[\log(C(F_{n,t}))=-\theta n^{-1}+\omega C'(1)n^{-2}\log(n)+\Bigl(\frac{2}{\sigma^4}C''(1)+C_tC'(1)-\frac{\theta^2}{2}\Bigr)n^{-2}+\bigo(n^{-3}\log^2(n)).\]
    Now, \(F_{k,t}>t\geq p\) for \(k\geq 1\), and so this implies \(\Pi_{n,t}>0\) (and well-defined). By the expansion above, we have that for some \(c_t\) dependent on \(t\)
    \[\lim_{n\to\infty}\frac{\Pi_{n,t}}{n^{-\theta}}=c_t>0.\]
    Let \(\gamma\) represent Euler's constant. By the expansion for the Harmonic Series, we have
    \begin{equation}\label{harmonic asymptotics}
        n^{-\theta}=\exp\Bigl(\sum_{k=1}^n-\theta k^{-1}\Bigr)e^{\theta\gamma+\theta/2+\bigo(n^{-2})}=e^{\theta\gamma}\Bigl(\prod_{k=1}^ne^{-\theta k^{-1}}\Bigr)\Bigl(1+\frac{\theta}{2}n^{-1}+\bigo(n^{-2})\Bigr).
    \end{equation}
    So as \(n\to\infty\),
    \[\frac{\Pi_{n,t}}{\prod_{k=1}^ne^{-\theta k^{-1}}}=\frac{\Pi_n}{n^{-\theta}e^{-\theta\gamma}(1+\bigo(n^{-1}))}\longrightarrow c_te^{\theta\gamma},\]
    which means
    \[\prod_{k=1}^\infty C(F_{k,t})e^{\theta k^{-1}}=(C(F_{0,t}))^{-1}c_te^{\theta\gamma}.\]
    We consider the tail of this product. We have
    \begin{equation*}
        \prod_{k\geq n+1}C(F_{k,t})e^{\theta k^{-1}}=\exp\Bigl(\sum_{k\geq n+1}\omega C'(1)\frac{\log(k)}{k^2}+\Bigl(\frac{2C''(1)}{\sigma^4}+C_tC'(1)-\frac{\theta^2}{2}\Bigr)k^{-2}+\bigo\bigl(\frac{\log^2(k)}{k^3}\bigr)\Bigr).
    \end{equation*}
    Now the Euler-Maclaurin formula provides \(\sum_{k\geq n+1}k^{-2}\log(k)=n^{-1}\log(n)+n^{-1}+\bigo(n^{-2}\log(n)),\ \sum_{k\geq n+1}k^{-2}=n^{-1}+\bigo(n^{-2})\), and \(\sum_{k\geq n+1}k^{-3}\log^2(k)=n^{-2}\log^2(n)+\bigo(n^{-2}\log(n))\). Applying  these formulae together with \(e^{x}=1+x+\bigo(x^2)\) yields 
    \begin{equation}\label{tail prod asymptotics}
        \prod_{k\geq n+1}C(F_{k,t})e^{\theta k^{-1}}=1+\omega C'(1)\frac{\log(n)}{n}+\Bigl(\frac{2C''(1)}{\sigma^4}+(C_t+\omega)C'(1)-\frac{\theta^2}{2}\Bigr)n^{-1}+\bigo\bigl(\frac{\log^2(n)}{n^2}\bigr).
    \end{equation}
    Finally, \[c_te^{\theta\gamma}=C(F_{0,t})\prod_{k=1}^\infty C(F_{k,t})e^{\theta k^{-1}}=\Bigl(\Pi_{n,t}\Bigr)\Bigl(\prod_{k=1}^ne^{\theta k^{-1}}\Bigr)\Bigl(\prod_{k\geq n+1}C(F_{k,t})e^{\theta k^{-1}}\Bigr).\]
    Solving for \(\Pi_{n,t}\) here and applying the asymptotic expansions of the latter two products provided in equations \eqref{harmonic asymptotics} and \eqref{tail prod asymptotics} gives the result.
\end{proof}
We introduce one more generating function that will help simplify our proof of Theorem \ref{thm tau}. Set \(\Psi(t)=\sum_{n\geq 0}\Pi_{n,p}t^n\). Recalling \(\Pi_{-1,p}=1\), we have the following important identity:
\begin{lemma}\label{a+b id}
    \(\cala(t)+\calb(t)=1-(1-t)\Psi(t)\) for \(0\leq t<1\).
\end{lemma}
\begin{proof}
By Remark \ref{a+b+c=1},
    \begin{align*}
    \cala(t)+\calb(t)&=\sum_{n\geq 0}(A(F_{n,p})+B(F_{n,p}))\Pi_{n-1,p}t^n\\
    &=\sum_{n\geq 0}(1-C(F_{n,p}))\Pi_{n-1,p}t^n=\sum_{n\geq 0}\Pi_{n-1,p}t^n-\sum_{n\geq 0}\Pi_{n,p}t^n=1-(1-t)\Psi(t).
\end{align*}
\end{proof}
It will also be helpful at times to write \(\calc(t)=\Psi(t)-\Psi_G(t)\), where \(\Psi_G(t)=\sum_{n\geq 0}(1-G(F_{n+1,p}))\Pi_{n,p}t^n\).
\begin{lemma}\label{several expansions}
Let \(c_p\) be the constant in the Lemma \ref{lemma prod c asymptotics} applied to the case \(t=p\). Under Assumption A, \(\cala(t),\Psi(t),\) and \(\Psi_G(t)\) have the following asymptotics as \(t\uparrow1\):
\begin{enumerate}
    \item[(i)] If \(\theta>1\), then \(|\cala(1)|,\Psi(1),\Psi_G(1)<\infty\).
    \item [(ii)] If \(\theta=1\), then \(|\cala(1)|,\Psi_G(1)<\infty\) and \(\Psi(t)\sim -c\log(1-t)\).
    \item[(iii)] If \(0<\theta<1\), then \(|\cala(1)|,\Psi_G(1)<\infty\) and \(\Psi(t)\sim c_p\Gamma(1-\theta)(1-t)^{\theta-1}\).
    \item[(iv)] Under Assumption B, if \(\theta=0\), then
    \begin{align*}
        &\cala(t)=\frac{2c_pA'(1)}{\sigma^2}\log(1-t)+\bigo(1),&&\cala'(t)=-\frac{2c_pA'(1)}{\sigma^2(1-t)}+\bigo(\log^2(1-t)),\\
         &\Psi_G(t)=-\frac{2c_pG'(1)}{\sigma^2}\log(1-t)+\bigo(1),&&\Psi_G'(t)=\frac{2c_pG'(1)}{\sigma^2(1-t)}+\bigo(\log^2(1-t)),\\
         &\Psi(t)=\frac{c_p}{1-t}+c_p\beta\log(1-t)+\bigo(1),&&\Psi'(t)=\frac{c_p}{(1-t)^2}-\frac{c_p\beta}{(1-t)}+\bigo(\log^3(1-t)).
    \end{align*}
    \item[(v)] If \(\theta<0\), \(\cala(t)\sim -\Gamma(-\theta)\frac{2c_p}{\sigma^2}A'(1)(1-t)^\theta,\ \Psi_G(t)\sim \Gamma(-\theta)\frac{2c_p}{\sigma^2}G'(1)(1-t)^\theta\), and \(\Psi(t)\sim c_p\Gamma(1-\theta)(1-t)^{\theta-1}\).
\end{enumerate}
\end{lemma}
\begin{proof}
We first prove (i)-(iii). We have \(A(x)=-(1-x)A'(1)+\bigo(1-x)^2\). Since \(I\) and \(J\) have finite second moments, so does \(Z_0\). This means \(1-G(x)=(1-x)G'(1)+\bigo(1-x)^2\), so by Proposition \ref{fn asymptotics}, \(A(F_{n,p})=-\frac{2}{\sigma^2}A'(1)n^{-1}+\bigo(n^{-2}\log(n))\) and \(1-G(F_{n+1,p})=\frac{2}{\sigma^2}G'(1)n^{-1}+\bigo(n^{-2}\log(n))\). By Lemma \ref{lemma prod c asymptotics}, \(\Pi_{n,p}=c_pn^{-\theta}(1+\bigo(n^{-1}\log(n)))\), which implies \(A(F_{n,p})\Pi_{n-1,p}\) and \((1-G(F_{n+1,p}))\Pi_{n,p}\) are of the order \(n^{-\theta-1}\), so clearly \(|\cala(1)|\) and \(\Psi_G(1)\) are both finite for \(\theta>0\). Likewise, \(\Psi(1)<\infty\) for \(\theta>1\). Now \(\Pi_{n,p}> 0\) for all \(n\) and \(\Pi_{n,p}\sim c_pn^{-\theta}\), so the Tauberian theorem in \cite[Theorem 13.5]{Fel71} implies \(\Psi(t)\) has the claimed asymptotics for \(0<\theta\leq 1\).

(iv) is a little more delicate. There exists a constant \(C\) such that \(|A(F_{n,p})\Pi_{n-1,p}+\frac{2c_p}{\sigma^2}A'(1)n^{-1}|\leq Cn^{-2}\log(n)\) and \(|(1-G(F_{n+1,p}))\Pi_{n,p}-\frac{2c_p}{\sigma^2}G'(1)n^{-1}|\leq Cn^{-2}\log(n)\) for all \(n \geq 1\), so \(\cala(t)\) and \(\Psi_G(t)\) both have the claimed asymptotics. As for their derivatives, \(\cala'(t)=t^{-1}\sum_{n\geq 1}nA(F_{n,p})\Pi_{n-1,p}t^n\) and \(|nA(F_{n,p})\Pi_{n-1,p}+\frac{2c_p}{\sigma^2}A'(1)|\leq Cn^{-1}\log(n)\). By standard asymptotics of the polylogarithm function, \(\sum_{n\geq 1}n^{-1}\log(n)t^n\sim \log^2(1-t)/2\), so \(\cala'(t)=t^{-1}\sum_{n\geq1}-\frac{2c_p}{\sigma^2}A'(1)t^n+\bigo(\log^2(1-t))=\frac{-2c_p}{\sigma^2}A'(1)(1-t)^{-1}+\bigo(\log^2(1-t))\) and similarly for \(\Psi_G(t)\). For \(\Psi(t)\), note that \(C'(1)=0\) when \(\theta=0\), and so there exists constant \(C\) such that \(|\Pi_{n,p}-(c_p-c_p\beta n^{-1})|\leq Cn^{-2}\log^2(n)\) for all \(n\geq 1\), and the claim follows. Finally, for \(\Psi'(t)\) we have \(|n\Pi_{n,p}-(c_pn-c_p\beta)|\leq Cn^{-1}\log^2(n)\) and \(\sum_{n\geq 1}n^{-1}\log^2(n)t^n\sim -\log^3(1-t)/3\). The estimates for \(\theta<0\) can be obtained similarly through comparison to the polylogarithm function.
\end{proof}
\begin{proof}[Proof of Theorem \ref{thm tau}]
We focus on \(\theta>0\) first, as the case \(\theta=0\) requires a separate argument. Note that the first derivative of \(t^2H(t)\) at \(t=1\) is \((I(p))^{-1}+1\), and so \(1-t^2H(t)-((I(p))^{-1}+1)(1-t)=\bigo(1-t)^2\). Using Lemma \ref{a+b id}, \(V(t)\) may then be rewritten as
\begin{equation}\label{v1}
    V(t)=(1-t)\Psi(t)+(1-t)+(I(p))^{-1}(1-t)\cala(t)+R_V(t)
\end{equation}
with \(R_V(t)=-(1-t)^2\Psi(t)+\left(1-t^2H(t)-((I(p))^{-1}+1)(1-t)\right)\cala(t)\). By Lemma \ref{several expansions}, we have
\[R_V(t)=\begin{cases}
    \bigo(1-t)^2&\text{if }\theta>1\\
    \bigo((1-t)^2\log(1-t))&\text{if }\theta=1\\
    \bigo(1-t)^{\theta+1}&\text{if }0<\theta<1.
\end{cases}\]
Similarly,
\begin{equation}\label{u1}
    U(t)=\Psi(t)-\Psi_G(t)+tH(t)G(p)+R_U(t)
\end{equation}
with \(R_U(t)=-((I(p))^{-1}+1)(1-t)\calc(t)+\left(t^2H(t)-(1-(1-t)((I(p))^{-1}+1))\right)\calc(t)\). Bearing in mind \(\calc(t)\) has the same leading order asymptotic as \(\Psi(t)\), again by Lemma \ref{several expansions}, we have
\[R_U(t)=\begin{cases}
    \bigo(1-t)&\text{if }\theta>1\\
    \bigo((1-t)\log(1-t))&\text{if }\theta=1\\
    \bigo(1-t)^{\theta}&\text{if }0<\theta<1.
\end{cases}\]
Now \[V(t)-(1-t)U(t)=(1-t)\Bigl(1+(I(p))^{-1}\cala(t)+\Psi_G(t)-tH(t)G(p)\Bigr)+R_V(t)-(1-t)R_U(t).\]
where the quantity in brackets tends to the  finite constant \(1+I^{-1}(p)\cala(1)+\Psi_G(1)-G(p)\) as \(t\uparrow1\). For \(\theta>0\), \(R_V(t)-(1-t)R_U(t)\) is dominated by the first term, so \(V(t)-(1-t)U(t)\sim(1-t)(1+I^{-1}(p)\cala(1)+\Psi_G(1)-G(p))\) as \(t\uparrow1\), assuming this constant is nonzero. To see this, first \(G(p)<1\) and \(\Psi_G(1)>0\) trivially. We also observe \(\cala(1)>0\) for \(\theta>0\). By Lemmas \ref{a+b id} and \ref{several expansions}, \(\cala(1)+\calb(1)=1\) for \(\theta>0\), and thus
\begin{align*}
    \cala(1)&=1-\calb(1)=1-I(p)\sum_{n\geq 0}\left(1-\frac{J(F_{n,p})(1-p)}{1-p/F_{n+1,p}}\right)\Pi_{n-1,p}\\
    &=1-I(p)\sum_{n\geq 0}(1-C(F_{n,p}))\Pi_{n-1,p}+I(p)\sum_{n\geq 0}\frac{J(F_{n,p})(1-p)}{1-p/F_{n+1,p}}\bigl(1-I(F_{n+1,p})\bigr)\Pi_{n-1,p}\\
    &=1-I(p)+I(p)\sum_{n\geq 0}\frac{J(F_{n,p})(1-p)}{1-p/F_{n+1,p}}\bigl(1-I(F_{n+1,p})\bigr)\Pi_{n-1,p}
\end{align*}
where in the second line we found the first sum equals \(\Pi_{-1,p}=1\) because it is telescoping with terms of the order \(n^{-\theta-1}\). Finally, \(I(p)<1\) and each of the three terms appearing in the second sum are positive, and thus \(\cala(1)=1-\calb(1)>0\). Now, using Lemma \ref{several expansions} and \eqref{v1}, \[V(t)\sim\begin{cases}
    (1-t)(1+\Psi(1)+(I(p))^{-1}\cala(1))&\text{if }\theta>1\\
    -c_p(1-t)\log(1-t)&\text{if }\theta=1\\
    c_p\Gamma(1-\theta)(1-t)^\theta&\text{if }0<\theta<1
\end{cases}\]
Setting \(C^*=1+(I(p))^{-1}\cala(1)+\Psi_G(1)-G(p)>0\), we have
\[Q(t)=\frac{V(t)-(1-t)U(t)}{(1-t)V(t)}\sim\begin{cases}
    C^*\bigl((1+\Psi(1)+(I(p))^{-1}\cala(1))(1-t)\bigr)^{-1}&\text{if }\theta>1\\
    C^*\bigl(-c_p\log(1-t)(1-t)\bigr)^{-1}&\text{if }\theta=1\\
    C^*\bigl(c_p\Gamma(1-\theta)\bigr)^{-1}(1-t)^{-\theta}&\text{if }0<\theta<1.
\end{cases}\]
Since \(P(Z_n>0)\) is decreasing in \(n\), the claimed asymptotics on \(P(Z_n>0)=P(\tau>n)\) follow from the Tauberian theorem in \cite[Theorem 13.5]{Fel71}. For \(\theta=0\), we cannot apply the Tauberian theorem to \(Q(t)\) directly. Instead, we will apply it to \(Q'(t)\) and use the following monotone density fact to get an asymptotic on \(P(Z_n>0)\), \(Q(t)\)'s coefficients:
\begin{proposition}\label{monotone density}
    Let \(a_k\) be an ultimately monotone sequence. If \(\sum_{k=1}^nka_k\sim n^p\) for some \(p>0\), then \(a_n\sim pn^{p-2}\).
\end{proposition}
We refer the reader to the appendix for the proof of Proposition \ref{monotone density}. We have
\[Q'(t)=\frac{V^2(t)-(1-t)(V(t)U'(t)-U(t)V'(t))}{(1-t)^2V^2(t)}\]
We first obtain asymptotics on \(V'(t)\) and \(V(t)\). Differentiating \eqref{v1} and applying Lemma \ref{several expansions}, we find
\[V'(t)=-c_pv_1\log(1-t)+v_2+o(1)\]
where \(v_1=\beta+\frac{2}{\sigma^2}(I(p))^{-1}A'(1)\) and \(v_2\) is another constant whose particular value will turn out not to matter. \eqref{v1} and Lemma \ref{several expansions} also imply \(V(1)=\lim_{t\to 1}V(t)=c_p\). Integrating the asymptotic on \(V'(t)\) yields in turn
\[V(t)=c_p+c_pv_1(1-t)\log(1-t)-(1-t)(c_pv_1+v_2)+o(1-t).\]
Turning to \(U'(t)\) and \(U(t)\), \eqref{u1} and Lemma \ref{several expansions} imply \[U(t)=c_p(1-t)^{-1}+c_pu\log(1-t)+\bigo(1)\] and \[U'(t)=c_p(1-t)^{-2}-c_pu(1-t)^{-1}+\bigo(\log^3(1-t))\]
with \(u=\beta+\frac{2}{\sigma^2}G'(1)\). Now we have
\[V(t)U'(t)-U(t)V'(t)=c_p^2(1-t)^2+2c_p^2v_1(1-t)^{-1}\log(1-t)-(c_p^2v_1+2c_pv_2+c_p^2u)(1-t)^{-1}+o(1-t)\]
and
\[V^2(t)=c_p^2+2c_p^2v_1(1-t)\log(1-t)-(2c_p^2v_1+2c_pv_2)(1-t)+o(1-t),\]
hence
\[V^2(t)-(1-t)^2(V(t)U'(t)-U(t)V'(t))\sim c_p^2(u-v_1)(1-t).\]
Noting that when \(\theta=0\), \(A'(1)=\frac{\theta\sigma^2}{2}(I(p)-1)-I(p)\mu=-I(p)\mu\), we have
\[Q'(t)\sim (u-v_1)(1-t)^{-1}=\frac{2}{\sigma^2}(G'(1)+\mu)(1-t)^{-1},\] which means \(\sum_{k=1}^n kP(Z_k>0)\sim \frac{2}{\sigma^2}(G'(1)+\mu)n\) by the Tauberian Theorem and \(P(Z_n>0)\sim \frac{2}{\sigma^2}(G'(1)+\mu)n^{-1}\) by Proposition \ref{monotone density}. Lastly, we note that \(G'(1)=E[Z_0]\).
\begin{remark}
If we set \(Y_0=y>0\), this argument also provides the estimate
\[P(T_0>n)\sim \frac{2}{\sigma^2}(y+\mu)n^{-1}.\]
The presence of \(\mu\) here is noteworthy. In the classic Galton-Watson process without migration, this same quantity is asymptotic to \(\frac{2y}{\sigma^2}n^{-1}\), which was first proved by Kolmogorov in 1938, see \cite{AN04}. Although one may first suspect migration should have a negligible effect on the process when \(\theta=0\), it actually allows it to survive a tiny bit longer.
\end{remark}
At last we turn to the \(\theta<0\) case. Altering which terms appear in the main asymptotic and which appear in the error term in equations \eqref{v1} and \eqref{u1} and using Lemma \ref{several expansions}, we obtain
\begin{align*}
    V(t)&=(1-t)\Psi(t)-(1-t)^2\Psi(t)+(1-t)(I(p))^{-1}\cala(t)+\bigo(1-t)\\
    U(t)&=\Psi(t)-\Psi_G(t)-((I(p))^{-1}+1)(1-t)\Psi(t)+\bigo(1)
\end{align*}
and so
\begin{align*}
    V(t)-(1-t)U(t)&=(I(p))^{-1}(1-t)^2\Psi(t)+(1-t)\Psi_G(t)+(1-t)(I(p))^{-1}\cala(t)+\bigo(1-t)\\
    &\sim c_p(1-t)^{\theta+1}\left((I(p))^{-1}\Gamma(1-\theta)+\frac{2}{\sigma^2}\Gamma(-\theta)\Bigl(G'(1)-(I(p))^{-1}A'(1)\Bigr)\right)\\
    &=c_p\Gamma(1-\theta)(1-t)^{\theta+1}\left(1-\frac{2}{\theta\sigma^2}(E[Z_0]+\mu)\right).
\end{align*}
Now \((1-t)V(t)\sim c_p\Gamma(1-\theta)(1-t)^{\theta+1}\), so this implies
\(Q(1)=\lim_{t\to 1}Q(t)=1-\frac{2}{\theta \sigma^2}(E[Z_0]+\mu)\), but \(Q(1)=E[\tau]\).
\end{proof}
\subsection{Proofs of the claims concerning the stationary distribution}
In this section, we present a functional equation that characterizes the stationary stationary distribution of \(Y_n\) when \(\theta<0\) and derive formulae for its mean and the weight it assigns to 0. The method is based on the idea in \cite{BS08}.
\begin{lemma}
    Let \(Y_n(t)=\sum_{k\geq 0}P(Y_n=k)t^k\) and \(n\geq 0\). Then for any \(t\) with \(F(t)\neq p\),
    \[Y_{n+1}(t)=B(t)Y_n(p)+C(t)Y_n(F(t))\]
    with \(B(t)\) and \(C(t)\) as in Lemma \ref{zn recurrence}.
\end{lemma}
The proof is similar to that of Lemma \ref{zn recurrence} and is therefore omitted. An important corollary follows immediately:
\begin{corollary}\label{y functional eq}
    Let \(\theta<0\). If \(Y(t)=\sum_{k\geq 0}\pi(k)t^k\), then for \(F(t)\neq p\)
    \[Y(t)=B(t)Y(p)+C(t)Y(F(t)).\]
\end{corollary}
\begin{proof}[Proofs of Theorems \ref{thm stationary mean} and \ref{pi 0}]
Let \(t_0=p\) and set \(t_{n+1}=F(t_n)\) for \(n\geq 0\). By our assumptions \(F(t)>p\) for any \(t\geq p\) and so \(t_{n+1}>p\) for all \(n\geq 0\). Then by Corollary \ref{y functional eq} we have
\[Y(t_n)=B(t_n)Y(p)+C(t_n)Y(t_{n+1}).\]
Recalling our earlier notation, \(\Pi_{n,p}=\prod_{k=0}^nC(t_k)\). We will write \(\Pi_n=\Pi_{n,p}\) going forward. Solving for \(Y(t_{n+1})\) and iterating yields
\[Y(t_{n+1})=Y(t_0)\Pi_n^{-1}-Y(p)\Pi_n^{-1}\sum_{k=0}^nB(t_k)\Pi_{k-1}.\]
The latter term simplifies slightly:
\begin{align*}
    \frac{1}{I(p)}\sum_{k=0}^nB(t_k)\Pi_{k-1}&=\sum_{k=0}^n\Bigl(1-\frac{C(t_k)}{I(t_{k+1})}\Bigr)\Pi_{k-1}\\
    &=\sum_{k=0}^n\Pi_{k-1}-\Pi_{k}+\Bigl(1-\frac{1}{I(t_{k+1})}\Bigr)\Pi_{k}=1-\Pi_{n}+\sum_{k=0}^n\Bigl(1-\frac{1}{I(t_{k+1})}\Bigr)\Pi_{k}.
\end{align*}
Thus,
\begin{equation}\label{ytn}
    Y(t_{n+1})=\Pi_{n}^{-1}\Bigl(Y(t_0)-Y(p)I(p)\Bigr)+Y(p)I(p)\Bigl(1+\Pi_{n}^{-1}\sum_{k=0}^n(I^{-1}(t_{k+1})-1)\Pi_{k}\Bigr)
\end{equation}
Now by Lemma \ref{lemma prod c asymptotics}, \(\Pi_k=c_pk^{-\theta}(1-\omega C'(1)k^{-1}\log(k)-\beta k^{-1}+\bigo(k^{-2}\log^2(k))\)
Letting \(\kappa=EI^2-\mu\), we have \(I(t)=1-(1-t)\mu+\frac{\kappa}{2}(1-t)^2+\bigo(1-t)^3\), and by Proposition \ref{fn asymptotics}
\begin{align*}
    t_{k+1}&=1-\frac{2}{\sigma^2}(k+1)^{-1}+\omega (k+1)^{-2}\log(k+1)+C_p(k+1)^{-2}+\bigo((k+1)^{-3}\log^2(k+1))\\
    &=1-\frac{2}{\sigma^2}k^{-1}+\omega k^{-2}\log(k)+\bigl(C_p+\frac{2}{\sigma^2}\bigr)k^{-2}+\bigo(k^{-3}\log^2(k)),
\end{align*}
thus \quad\(\displaystyle I^{-1}(t_{k+1})=1+\frac{2\mu}{\sigma^2}k^{-1}-\omega \mu k^{-2}\log(k)-\Bigl(C_p\mu +\frac{2\mu}{\sigma^2}(1-\frac{2\mu}{\sigma^2})+\frac{2\kappa}{\sigma^4}\Bigr)k^{-2}+\bigo(k^{-3}\log^2(k)).\)
Using Lemma \ref{lemma prod c asymptotics} and \(C'(1)=\theta\sigma^2/2\) we have
\[(I^{-1}(t_{k+1})-1)\Pi_{k}=c_pk^{-\theta}(ak^{-1}-bk^{-2}\log(k)-ck^{-2}+\bigo(k^{-3}\log^2(k)))\]
with \(a=\frac{2\mu}{\sigma^2},\ b=\mu\omega(\theta+1),\) and \(c=C_p\mu +\frac{2\mu}{\sigma^2}(1-\frac{2\mu}{\sigma^2})+\frac{2\kappa}{\sigma^4}+\frac{2\mu\beta}{\sigma^2}\) and \(\beta\) and \(\omega\) as in Lemma \ref{lemma prod c asymptotics}. By Euler-Maclaurin summation, we get
\[\sum_{k=2}^nk^{\varepsilon}\log^\upsilon(k)=\begin{cases}
    \frac{n^{\varepsilon+1}}{\varepsilon+1}+\frac{n^{\varepsilon}}{2}+\bigo(n^{\varepsilon-1}+ 1)&\text{if }\varepsilon\neq -1,\ \upsilon=0\\
    \frac{n^{\varepsilon+1}\log(n)}{\varepsilon+1}-\frac{n^{\varepsilon+1}}{(\varepsilon+1)^2}+\bigo(n^{\varepsilon}\log(n)+1)&\text{if }\varepsilon\neq -1,\ \upsilon=1\\
    \frac{n^{\varepsilon+1}\log^2(n)}{\varepsilon+1}+\bigo(n^{\varepsilon+1}\log(n)+1)&\text{if }\varepsilon\neq -1,\ \upsilon=2\\
    \frac{\log^{\upsilon+1}(n)}{\upsilon+1}+\bigo(1)&\text{if }\varepsilon=-1,\ \upsilon\neq -1.\\
\end{cases}\]
This implies that for \(\theta<-1\)
\[\sum_{k=0}^n\Bigl(\frac{1}{I(t_{k+1})}-1\Bigr)\Pi_{k}=\frac{c_p}{n^\theta}\Bigl(\frac{-a}{\theta}+\frac{b}{\theta+1}\frac{\log(n)}{n}+\bigl(\frac{a}{2}+\frac{b}{(\theta+1)^2}+\frac{c}{\theta+1}\bigr)\frac{1}{n}\Bigr)+\bigo\bigl(\frac{\log^2(n)}{n^{\theta+2}}+1\bigr),\]
and that for \(\theta=-1\)
\begin{equation}\label{theta=-1}
    \sum_{k=0}^n(I^{-1}(t_{k+1})-1)\Pi_{k}=c_pn\Bigl(a-cn^{-1}\log(n)\Bigr)+\bigo(1),
\end{equation}
since \(b=0\) when \(\theta=-1\). By Lemma \ref{lemma prod c asymptotics}, we have
\begin{equation}\label{prod inverse asymptotics}
    \Pi_{n}^{-1}=c_p^{-1}n^{\theta}\bigl(1+\omega C'(1)n^{-1}\log(n)+\beta n^{-1}\bigr)+\bigo(n^{\theta-2}\log^2(n)),
\end{equation}
and so for \(\theta<-1\),
\begin{equation}\label{sum asymptotics}
\begin{aligned}
    \Pi_{n}^{-1}\sum_{k=0}^n(I^{-1}(t_{k+1})-1)&\Pi_{k}=-\frac{a}{\theta}+\Bigl(\frac{b}{\theta+1}-\frac{\omega aC'(1)}{\theta}\Bigr)n^{-1}\log(n)\\
    +&\Bigl(\frac{a}{2}+\frac{b}{(\theta+1)^2}+\frac{c}{\theta+1}-\frac{a\beta}{\theta}\Bigr)n^{-1}+\bigo(n^{-2}\log^2(n)+n^{-\theta}).
\end{aligned}
\end{equation}
Substituting \eqref{prod inverse asymptotics} and \eqref{sum asymptotics} into \eqref{ytn} and taking limits, we observe the first term is order \(n^{\theta}\) and disappears for \(\theta<0\), so:
\begin{equation}\label{ypip}
    1=\lim_{n\to\infty} Y(t_{n+1})=Y(p)I(p)\Bigl(1-\frac{a}{\theta}\Bigr)
\end{equation}
and thus \begin{align*}
    1-Y(t_{n+1})&=-Y(p)I(p)\Bigl(\frac{b}{\theta+1}-\frac{\omega a C'(1)}{\theta}\Bigr)n^{-1}\log(n)\\
    &\quad-Y(p)I(p)\Bigl(\frac{a}{2}+\frac{b}{(\theta+1)^2}+\frac{c}{\theta+1}-\frac{a\beta}{\theta}\Bigr)n^{-1}+\bigo(n^{-2}\log^2(n)+n^{-\theta}).
\end{align*}
Now \(\frac{b}{\theta+1}-\frac{\omega aC'(1)}{\theta}=0\) is easily computed. Now \(Y'(t)\) is increasing in \(t\), and so \(\lim_{t\uparrow 1}Y'(t)=Y'(1)\in [0,\infty]\) exists, so it can be computed along the sequence \(\{t_n\}_{n\geq 0}\). Since \(t_{n+1}=1-\frac{2}{\sigma^2}n^{-1}+\bigo(n^{-2}\log(n))\), we have
\[Y'(1)=\lim_{n\to\infty}\frac{1-Y_{n+1}}{1-t_{n+1}}=-\frac{\sigma^2}{2}Y(p)I(p)\Bigl(\frac{a}{2}+\frac{b}{(\theta+1)^2}+\frac{c}{\theta+1}-\frac{a\beta}{\theta}\Bigr).\]
Using \eqref{ypip}, we also find \(Y(p)I(p)=\frac{\mu+\nu-\rho}{\nu-\rho}\). Corollary \ref{y functional eq} gives \(Y(0)=Y(p)I(p)\), and so as a bonus we get \(\pi(0)=\frac{\mu+\nu-\rho}{\nu-\rho}\), proving Theorem \ref{pi 0}.\footnote{This is not immediate. Theorem \ref{pi 0} is stated to hold under Assumption A, but above we have assumed B here. Assuming A only merely replaces the asymptotic \eqref{sum asymptotics} with \(-\frac{a}{\theta}+o(1)\), valid for any \(\theta<0\).} We spare the reader of the tedious computation involved in simplifying \(\frac{a}{2}+\frac{b}{(\theta+1)^2}+\frac{c}{\theta+1}-\frac{a\beta}{\theta}\). The constant \(C_p\) disappears, as it should or else this would strongly indicate a contradiction, since \(C_p\) a priori depends on the initial point \(t_0=p\), which was chosen for convenience only. The constant \(\omega\) vanishes too, meaning \(\bar{\pi}\) is independent of \(E[X^3]\). Using
\(C''(1)=\Var{I+J+K}+(\mu+\nu-\rho)(\sigma^2-1)+(\mu+\nu-\rho)^2-\sigma^2\nu\), we find
\[\frac{a}{2}+\frac{b}{(\theta+1)^2}+\frac{c}{\theta+1}-\frac{a\beta}{\theta}=\frac{\mu}{\sigma^2(\theta+1)}\left(\theta-\frac{2\mu}{\sigma^2}+\frac{4\nu}{\sigma^2\theta}+\frac{2\Var{I}}{\mu\sigma^2}-\frac{4\Var{I+J+K}}{\sigma^4\theta}\right)\]
and hence
\[\Bar{\pi}=Y'(1)=\frac{-\theta\mu}{4(\theta+1)(\nu-\rho)}\left(\theta\sigma^2-2\mu+\frac{4\nu}{\theta}+\frac{2\Var{I}}{\mu}-\frac{4\Var{I+J+K}}{\sigma^2\theta}\right).\] 
Finally, it remains to show that \(\Bar{\pi}=\infty\) for \(-1\leq \theta<0\). By the following proposition, it will be enough to establish this for \(\theta=-1\):
\begin{proposition}\label{y monotonicity}
    Let \(Y_n^{\rho}\) be the Markov chain \(Y_n\) with emigration component \(K_n^\rho\), a sequence of geometric random variables with mean \(\rho\). If \(\rho_1\leq \rho_2\) and \(Y_0^{\rho_1}=Y_0^{\rho_2}\), then there is a coupling such that \(Y_n^{\rho_2}\leq Y_n^{\rho_1}\) for all \(n\)
\end{proposition}
\begin{proof}
    If \(G_q,G_q'\sim\geo{q}\), \(G_p\sim\geo{p}\), and \(\xi\sim\ber{1-q/p}\) are all independent and \(q\leq p\), then \(G_q\overset{d}{=}G_p+\xi(1+G_q')\). So there is a coupling with \(K_n^{\rho_1}\leq K_n^{\rho_2}\) for all \(n\). Fixing all other random variables appearing in the definition of \(Y_n\), \eqref{def yn}, we see \(Y_n\) stochastically decreases as emigration increases, giving the claim.
\end{proof}
Now if \(\rho_1\) is such that \(-1<\theta_1<0\) one can increase emigration to the level \(\rho_2=\mu+\nu+\frac{\sigma^2}{2}\) while fixing \(I\) and \(J\). This latter process has the critical parameter \(\theta_2=-1\), and by the proposition is stochastically dominated by the process having parameter \(\theta_1\). To see that \(\Bar{\pi}=\infty\) when \(\theta=-1\),
by \eqref{theta=-1} and \eqref{prod inverse asymptotics},
\[\Pi_{n}^{-1}\sum_{k=0}^n(I^{-1}(t_{k+1})-1)\Pi_{k}=a+(a\omega C'(1)-c)n^{-1}\log(n)+\bigo(n^{-1}).\]
This constant simplifies as
\begin{align*}
    a\omega C'(1)-c&=\frac{\mu}{\sigma^2}+\frac{2\mu}{\sigma^4}-\frac{2\Var{I}}{\sigma^4}+\frac{4\mu\nu}{\sigma^4}-\frac{4\mu\Var{I+J+K}}{\sigma^6}\\
    &\leq \frac{\mu}{\sigma^2}\Bigl(1+\frac{2+4\nu}{\sigma^2}-\frac{4\Var{K}}{\sigma^4}\Bigr)=-\frac{4\mu}{\sigma^6}\Bigl((\mu+\nu)^2+\mu+\nu+2\sigma^2\Bigr)<0.
\end{align*}
Repeating the earlier argument, \(\frac{1-Y_{n+1}}{1-t_{n+1}}\) is now of the order \(\log(n)\), and so \(\Bar{\pi}=\infty\).
\end{proof}

\section{Speed and Escape Probability for HYC}
\begin{proof}[Proof of Theorems \ref{thm hyc escape} and \ref{thm hyc speed}]
Given Theorem \ref{thm stationary mean}, the argument is nothing novel, see \cite{BS08} for another example of it. Recall the well-known identity for transient, nearest-neighbor walks on \(\Z\): \[\lim_{n\to\infty}\frac{X_n}{n}=v \text{ a.s.}\iff\lim_{n\to\infty}\frac{T_n}{n}=\frac{1}{v}\text{ a.s.}\]
which holds for \(v=0\) as well if we set \(1/0=\infty\). Theorem \ref{pinsky speed} already implies the a.s. existence of these limits. When \(T_n<\infty\), we have \(T_n=n+2\sum_{k\leq n}\lj{k}{n}.\)
For \(p>2/3\), the walk is transient and so with probability 1 there exists a \(C\) such that \(\sum_{k\leq -1}\lj{k}{n}\leq C\), and so with probability 1
\[\frac{T_n}{n}=1+\frac{2}{n}\sum_{k=0}^n\lj{k}{n}+\bigo(n^{-1}).\] Now, by the argument in Section 2, \(\{\lj{n}{n},\ldots,\lj{0}{n}\}\overset{d}{=}\{V_0,\ldots,V_n\}\). The process \(V_n\), which is clearly irreducible and aperiodic, satisfies Assumption B with \(I=J=1\), \(K\sim\geo{1-p}\), and \(\sigma^2=2\). We calculate \(\theta=2-\rho=\frac{2-3p}{1-p}\). For \(p>2/3\), \(\theta<0\) and thus \(V_n\) is positive recurrent. Let \(\pi\) denote its stationary distribution. By Birkhoff's ergodic theorem, \(n^{-1}\sum_{k=0}^n V_k\rightarrow \Bar{\pi}\) a.s., and so \(v=v(p)=(1+2\Bar{\pi})^{-1}\). We have \(-1\leq \theta<0\) when \(2/3<p\leq 3/4\), and so by Theorem \ref{thm stationary mean}, \(\Bar{\pi}=\infty\) and \(v(p)=0\), which was already shown in \cite{Pin10} and is also contained in \cite{KP17}. If \(p>3/4\), after some simplifications Theorem \ref{thm stationary mean} yields
\[\Bar{\pi}=-\frac{2}{\theta+1}=\frac{2(1-p)}{4p-3},\]
and so for \(p> 3/4\), \(v(p)=4p-3\). Proposition \ref{y monotonicity} implies \(v(p)\) is nondecreasing in \(p\), and thus \(v(3/4)=0\). For \(p>2/3\) we also wish to compute the escape probability \(P_1(T_0=\infty)=P(T_{-1}=\infty)\). We claim that \(P(T_{-1}=\infty)=\pi(0)\) for \(p>2/3\). To see why, when \(T_n<\infty\), \(\{T_{-1}>T_n\}=\{\lj{0}{n}=0\}\). Now with probability 1, \(T_n<\infty\) for all \(n\geq 0\) and \(T_{n+1}>T_n\geq n\). Since \(V_n\) is aperiodic, \[P(T_{-1}=\infty)=\lim_{n\to\infty}P(T_{-1}>T_n)=\lim_{n\to\infty}P(\lj{0}{n}=0)=\lim_{n\to\infty}P(V_n=0)=\pi(0).\]
By Theorem \ref{pi 0}, \(\pi(0)=\frac{\mu+\nu-\rho}{\nu-\rho}\). Setting \(\mu=\nu=1\) gives \(P_1(T_0=\infty)=\frac{3p-2}{2p-1}\) for \(p>2/3\).
\end{proof}

\section{The \((p,q)\)-HYC random walk}
Consider the following modification to the HYC random walk: The walk starts at 0 and steps right from site \(x\) with probability \(p\in (1/2,1)\) until it steps left. After taking a left step, it steps left with probability \(1-q\) until it steps right again from \(x\), thereafter stepping left and right from \(x\) with equal probabilities. We will call this a \((p,q)\)-HYC walk. Much like the HYC walk, this can be viewed as a cookie random walk with a Markovian cookie stack, so using the formula in \cite{KP17}, one can determine when this walk is transient and ballistic. One calculates the relevant parameter of \(\delta=\frac{p+q-1}{(1-p)q}\). When \(\delta>2\), the walk is ballistic (\(\lim_{n\to\infty}X_n/n\) exists a.s. and is positive) by \cite[Theorem 1.7]{KP17}. Using the same notation as in Section 3 and borrowing the argument appearing there we find \(\{\lj{n}{n},\ldots,\lj{0}{n}\}\overset{d}{=}\{V_0',\ldots,V_n'\}\) with \[V_{n+1}'=\Bigl(1+G_{q,n}+\sum_{m=1}^{V_n'-K_n}X_{m,n}\Bigr)\ind{K_n<V_n'+1}\]
and \(G_{q,n}\overset{iid}{\sim} \geo{q},\ K_n\overset{iid}{\sim}\geo{1-p},\) and \(X_{m,n}\overset{iid}{\sim} \geo{1/2}\), all independent of each other. Suppose \(q\leq 1/2\) so that the walk becomes excited to the left after taking a left step. Recall the coupling \(G_q\overset{d}{=}G_{1/2}+\xi(1+G_q')\) when \(G_q,G_q'\sim\geo{q}\), \(G_{1/2}\sim \geo{1/2}\), and \(\xi\sim\ber{1-2q}\) are all independent and \(q\leq 1/2\). This means our above branching process is equivalent in distribution to the following:
\[V_{n+1}'=\Bigl(1+\xi_n(1+G_{q,n})+\sum_{m=1}^{V_n'-K_n+1}X_{m,n}\Bigr)\ind{K_n<V_n'+1}\]
Setting \(J_n=1+\xi_n(1+G_{q,n})\) and  \(I_n=1\), this process is irreducible, aperiodic, and satisfies Assumption B. Moreover, it only differs from the HYC case in its non-reproducing immigrants, \(J_n\). We calculate \(\theta=q^{-1}-\rho=-\delta+1\). When \(\delta>2\), which is the same as \(p>\frac{1+q}{1+2q}\), \(\theta<-1\) and so Theorem \ref{thm stationary mean} implies this process has a stationary distribution \(\pi'\) with mean
\[\overline{\pi'}=\frac{q(1-p)(2p-1)}{(p+q-1)(2pq+p-q-1)}\]
and so by the earlier argument the speed of this random walk is \((1+2\overline{\pi'})^{-1}\) when \(q\leq 1/2\) and \(p>\frac{1+q}{1+2q}\). Similarly, for \(\theta<0\) or \(p>(1+q)^{-1}\), the escape probability for this random walk is \[P(T_{-1}=\infty)=\frac{p(1+q)-1}{p+q-1}.\]
Additional random walks can be cooked up whose Markov chain \(\{\lj{n}{n},\ldots,\lj{0}{n}\}\) is equivalent in distribution to something having the form \eqref{def yn}. For example, one can apply another rule on top of the \((p,q)\)-HYC walk: The moment the walk steps left from location \(x\), a ``toll" of \(M\) additional left steps from \(x\) is required before right steps are permitted from \(x\) with probability \(q\). After taking a right step, the site \(x\) becomes neutral again as in the \((p,q)\)-HYC walk. This merely changes the above non-reproducing immigration distribution to \(J=1+M+\xi(1+G_q)\) and our theorems apply.

\section{Appendix}
\noindent \textit{The HYC walk is a cookie random walk with Markovian cookies.} We briefly illustrate how the HYC walk can be interpreted as a cookie random walk with a Markovian cookie stack, a model introduced in \cite{KP17}. A similar argument shows the \((p,q)\)-HYC walk (with or without an additional toll) also fits into this framework. Using the notation in \cite{KP17}, for the HYC walk there are cookies with strength \(\mathbf{p}=[1,0,1/2]^T\). The initial distribution on these cookies is \(\eta=(p,1-p,0)\), since the first time the walk reaches a site it steps right with probability \(p\) and left with probability \(1-p\). The cookie stack's transitions are governed by the matrix \[K=\begin{pmatrix}
    p&1-p&0\\
    0&0&1\\
    0&0&1
\end{pmatrix},\]
as if the walk steps right from \(x\), it steps right from it on the next visit with probability \(p\) and left with probability \(1-p\). After stepping left from \(x\), the walk steps left and right from \(x\) with equal probability on all future visits, and so the cookie at that site takes strength \(1/2\).
\begin{proof}[Proof of Proposition \ref{fn asymptotics}]
    The claim follows from the argument in \cite[Chapters 8.5, 8.6]{deBru58}. If a positive sequence \(\{u_n\}_{n\geq 1}\) with \(u_n\to 0\) satisfies \(u_{n+1}=u_n-au_n^2+bu_n^3+\bigo(u_n^4)\) for \(a>0\), de Bruijn's trick implies \(u_n=a^{-1}n^{-1}-(a^2-b)a^{-3}n^{-2}\log(n)+C_1n^{-2}+\bigo(n^{-3}\log^2(n))\), with \(C_1\) dependent on \(u_1\). To see this, let \(y_n=u_n^{-1}\). Since \(y_n^{-1}\to 0\), we have
    \[y_{n+1}=y_n\bigl(1-(ay_n^{-1}-by_n^{-2}+\bigo(y_n^{-3}))\bigr)^{-1}=y_n+a+(a^2-b)y_n^{-1}+\bigo(y_n^{-2})\]
    and so
    \(\frac{a}{2}\leq y_{n+1}-y_n\leq 2a\) for all sufficiently large \(n\), so
    \[y_n=y_1+a(n-1)+\sum_{k=1}^{n-1}(y_{k+1}-y_k-a)=an+\bigo(\log(n))\]
    Define \(r_n=y_n-an\), so that \(r_n=\bigo(\log(n))\). \(r_n\) satisfies
    \[r_{n+1}=r_n+(a^2-b)a^{-1}n^{-1}(1+a^{-1}n^{-1}r_n)^{-1}+\bigo(n^{-2})=r_n+(a^2-b)a^{-1}n^{-1}+\bigo(n^{-2}\log(n)).\]
    Letting \(w_n=r_n-(a^2-b)a^{-1}\log(n)\), we similarly deduce
    \[w_{n+1}=w_n+\bigo(n^{-2}\log(n)).\]
    and hence \(w_n\to W=w_1+\sum_{k=1}^\infty w_{k+1}-w_k\) and \(W-w_n=\sum_{k\geq n}w_{k+1}-w_k=\bigo(n^{-1}\log(n))\). Finally
    \[y_n=an+(a^2-b)a^{-1}\log(n)+W+\bigo(n^{-1}\log(n))\]
    Inverting both sides of this equation gives the claim.
    Now \(F_{n,0}\) represents the probability that a critical Galton-Watson process with offspring distribution \(X\) and one initial member has died out by time \(n\), so this tends to 1 as \(n\to\infty\) and thus so will \(F_{n,t}\) for \(0\leq t\leq 1\). Now if \(E[X^4]<\infty\), \(F(x)=x+\frac{\sigma^2}{2}(1-x)^2-\frac{\alpha-3\sigma^2-1}{6}(1-x)^3+\bigo(1-x)^4\) with \(\alpha=E[X^3]\), and so letting \(u_n=1-F_{n,t}\), we have \(u_{n+1}=u_n-\frac{\sigma^2}{2}u_n^2+\frac{\alpha-3\sigma^2-1}{6}u_n^3+\bigo(u_n^4)\). If only \(E[X^3]<\infty\) is assumed, \(F(x)=x+\frac{\sigma^2}{2}(1-x)^2+\bigo(1-x)^3\), and a similar argument to the one above yields \(F_{n,t}=1-\frac{2}{\sigma^2}n^{-1}+\bigo(n^{-2}\log(n))\).
\end{proof}
\begin{proof}[Proof of Proposition \ref{monotone density}]
The proof is a minor adjustment to the proof of the monotone density theorem. We may assume \(a_k\) is ultimately decreasing, or else this follows directly from that theorem. Fix \(\lambda>1\). Let \(n\) be large enough so that \(a_k\) is decreasing and positive for all \(k\geq n/2\). Observe that if \(S_n=\sum_{k=1}^nka_k\), then
\[S_{\fl{\lambda n}}-S_n=\sum_{k={n+1}}^{\fl{\lambda n}}ka_k\leq (\lambda-1)n\Bigl(\max_{n+1\leq k\leq \fl{\lambda n}}ka_k\Bigr)\leq (\lambda-1)\lambda n^2a_{n},\]
including when \(\fl{\lambda n}=n\). Now we have \[\lambda(\lambda-1)n^{2-p}a_n\geq \frac{S_{\fl{\lambda n}}}{(\fl{\lambda n})^p}\frac{(\fl{\lambda n})^p}{n^p}-\frac{S_n}{n^p}.\]
Taking limits and using the assumption gives
\[\liminf_{n\to\infty}n^{2-p}a_n\geq \frac{\lambda^p-1}{\lambda(\lambda-1)}.\]
Sending \(\lambda \to 1\) yields \(\liminf_{n\to\infty}n^{2-p}a_n\geq p\). The upper bound is obtained similarly by taking \(1/2<\lambda<1\) and considering \(S_n-S_{\fl{\lambda n}}\).
\end{proof}

\section*{Acknowledgments}
The author would like to thank Jonathon Peterson for all his helpful feedback and pointing out the connection between the HYC walk and excited random walks with Markovian cookie stacks.

\bibliographystyle{amsplain}
\bibliography{citations.bib}

@article {KP17,
    AUTHOR = {Kosygina, Elena and Peterson, Jonathon},
     TITLE = {Excited random walks with {M}arkovian cookie stacks},
   JOURNAL = {Ann. Inst. Henri Poincar\'e{} Probab. Stat.},
  FJOURNAL = {Annales de l'Institut Henri Poincar\'e{} Probabilit\'es et
              Statistiques},
    VOLUME = {53},
      YEAR = {2017},
    NUMBER = {3},
     PAGES = {1458--1497},
      ISSN = {0246-0203,1778-7017},
   MRCLASS = {60K37 (60F05 60J10 60K35)},
  MRNUMBER = {3689974},
MRREVIEWER = {Bruno\ Schapira},
       DOI = {10.1214/16-AIHP761},
       URL = {https://doi.org/10.1214/16-AIHP761},
}

@article {Pin10,
    AUTHOR = {Pinsky, Ross G.},
     TITLE = {Transience/recurrence and the speed of a one-dimensional
              random walk in a ``have your cookie and eat it'' environment},
   JOURNAL = {Ann. Inst. Henri Poincar\'e{} Probab. Stat.},
  FJOURNAL = {Annales de l'Institut Henri Poincar\'e{} Probabilit\'es et
              Statistiques},
    VOLUME = {46},
      YEAR = {2010},
    NUMBER = {4},
     PAGES = {949--964},
      ISSN = {0246-0203,1778-7017},
   MRCLASS = {60K35 (60G50 60K37)},
  MRNUMBER = {2744879},
MRREVIEWER = {Quansheng\ Liu},
       DOI = {10.1214/09-AIHP331},
       URL = {https://doi.org/10.1214/09-AIHP331},
}

@book {Fel71,
    AUTHOR = {Feller, William},
     TITLE = {An introduction to probability theory and its applications.
              {V}ol. {II}},
   EDITION = {Second},
 PUBLISHER = {John Wiley \& Sons, Inc., New York-London-Sydney},
      YEAR = {1971},
     PAGES = {xxiv+669},
   MRCLASS = {60.00},
  MRNUMBER = {270403},
}

@incollection {YY95,
    AUTHOR = {Yanev, George P. and Yanev, Nickolay M.},
     TITLE = {Critical branching processes with random migration},
 BOOKTITLE = {Branching processes ({V}arna, 1993)},
    SERIES = {Lect. Notes Stat.},
    VOLUME = {99},
     PAGES = {36--46},
 PUBLISHER = {Springer, New York},
      YEAR = {1995},
      ISBN = {0-387-97989-1},
   MRCLASS = {60J80},
  MRNUMBER = {1351259},
       DOI = {10.1007/978-1-4612-2558-4\_5},
       URL = {https://doi.org/10.1007/978-1-4612-2558-4_5},
}

@article {KKS75,
    AUTHOR = {Kesten, H. and Kozlov, M. V. and Spitzer, F.},
     TITLE = {A limit law for random walk in a random environment},
   JOURNAL = {Compositio Math.},
  FJOURNAL = {Compositio Mathematica},
    VOLUME = {30},
      YEAR = {1975},
     PAGES = {145--168},
      ISSN = {0010-437X,1570-5846},
   MRCLASS = {60J15 (60F05 60J80)},
  MRNUMBER = {380998},
MRREVIEWER = {K.\ B.\ Erickson},
}

@article{Let21,
author = "Zachary A Letterhos",
title = "{Self-Interacting Random Walks and Related Braching-Like Processes}",
year = "2021",
month = "7",
url = "https://hammer.purdue.edu/articles/thesis/Self-Interacting_Random_Walks_and_Related_Braching-Like_Processes/15078462",
doi = "10.25394/PGS.15078462.v1"
}

@article {BS08,
    AUTHOR = {Basdevant, Anne-Laure and Singh, Arvind},
     TITLE = {On the speed of a cookie random walk},
   JOURNAL = {Probab. Theory Related Fields},
  FJOURNAL = {Probability Theory and Related Fields},
    VOLUME = {141},
      YEAR = {2008},
    NUMBER = {3-4},
     PAGES = {625--645},
      ISSN = {0178-8051,1432-2064},
   MRCLASS = {60K35 (60F15 60J80)},
  MRNUMBER = {2391167},
MRREVIEWER = {David\ A.\ Croydon},
       DOI = {10.1007/s00440-007-0096-8},
       URL = {https://doi.org/10.1007/s00440-007-0096-8},
}

@article{BS08-2,
author = {Anne-Laure Basdevant and Arvind Singh},
title = {{Rate of growth of a transient cookie random walk}},
volume = {13},
journal = {Electronic Journal of Probability},
number = {none},
publisher = {Institute of Mathematical Statistics and Bernoulli Society},
pages = {811 -- 851},
year = {2008},
doi = {10.1214/EJP.v13-498},
URL = {https://doi.org/10.1214/EJP.v13-498}
}

@article {BW03,
    AUTHOR = {Benjamini, Itai and Wilson, David B.},
     TITLE = {Excited random walk},
   JOURNAL = {Electron. Comm. Probab.},
  FJOURNAL = {Electronic Communications in Probability},
    VOLUME = {8},
      YEAR = {2003},
     PAGES = {86--92},
      ISSN = {1083-589X},
   MRCLASS = {60G50 (60K37)},
  MRNUMBER = {1987097},
MRREVIEWER = {Michael\ Voit},
       DOI = {10.1214/ECP.v8-1072},
       URL = {https://doi.org/10.1214/ECP.v8-1072},
}

@book{deBru58,
  title     = {Asymptotic Methods in Analysis},
  author    = {de Bruijn, Nicolaas Govert},
  year      = {1958},
  publisher = {North-Holland Publishing Company},
  address   = {Amsterdam}
}

@article{KMP21,
  title={Convergence of random walks with Markovian cookie stacks to Brownian motion perturbed at extrema},
  author={Kosygina, Elena and Mountford, Thomas and Peterson, Jonathon},
  journal={Annals of Probability},
  volume={49},
  number={6},
  pages={3120--3179},
  year={2021},
  publisher={Institute of Mathematical Statistics},
  url={https://arxiv.org/abs/2008.06766}
}

@book{GdPY18,
  title     = {Controlled Branching Processes},
  author    = {Gonz{\'a}lez, Miguel and del Puerto, In{\'e}s M. and Yanev, George P.},
  year      = {2018},
  publisher = {Wiley-ISTE},
  isbn      = {978-1786302533},
  pages     = {232}
}

@book{AN04,
  title={Branching Processes},
  author={Athreya, Krishna B. and Ney, Peter E.},
  year={2004},
  publisher={Dover Publications},
  address={Mineola, NY},
  isbn={978-0486435308}
}

@article{GMdP05,
 ISSN = {00219002},
 URL = {http://www.jstor.org/stable/30040804},
 author = {M. González and M. Molina and I. Del Puerto},
 journal = {Journal of Applied Probability},
 number = {2},
 pages = {463--477},
 publisher = {Applied Probability Trust},
 title = {Asymptotic Behaviour of Critical Controlled Branching Processes with Random Control Functions},
 urldate = {2026-08-22},
 volume = {42},
 year = {2005}
}

@article{BGMCdP24,
  author    = {Barczy, Matyas and Gonz{\'a}lez, Miguel and Mart{\'i}n-Ch{\'a}vez, Pedro and del Puerto, In{\'e}s},
  title     = {Diffusion approximation of critical controlled multi-type branching processes},
  journal   = {Revista de la Real Academia de Ciencias Exactas, F{\'i}sicas y Naturales. Serie A. Matem{\'a}ticas (RACSAM)},
  volume    = {118},
  number    = {3},
  articleno = {117},
  year      = {2024},
  doi       = {10.1007/s13398-024-01593-0},
  url       = {https://link.springer.com/article/10.1007/s13398-024-01593-0}
}

\end{document}